\documentclass[ECP,preprint]{ejpecp}

\usepackage{mathrsfs}

\SHORTTITLE{A Kolmogorov fourth-moment bound via a martingale core}
\TITLE{A Kolmogorov fourth-moment bound on Poisson chaos \\
via a martingale core}

\AUTHORS{Guangqu~Zheng\footnote{Department of Mathematics and
Statistics, Boston University, 665 Commonwealth Avenue,
Boston, MA 02215, USA. \EMAIL{gzheng90@bu.edu}}}

\AMSSUBJ{60F05 ; 60H07}
\AMSSUBJSECONDARY{60G55 ; 60H05}

\SUBMITTED{July 30, 2026}
\VOLUME{0}
\YEAR{2026}
\PAPERNUM{0}

\ABSTRACT{For any finite family of Poisson multiple integrals
and any finite $p\geq2$, we construct a common increasing
filtration generated by finitely many exact Poisson
counts such that the associated conditional
expectations converge in $L^p$, remain in their
original chaoses, and have bounded step kernels with
finite-measure support.  This finite-count martingale
core allows regular fixed-chaos identities and
estimates to be extended under the sole assumption of
a finite fourth moment.  In particular, if
$F$ lives in a   Poisson chaos with unit variance 
and finite fourth moment, 
we prove that the Kolmogorov distance 
between $F$ and a standard normal is bounded by 
$15.6\sqrt{\E[F^4]-3}$.
This removes Assumptions $\mathbf A$ and
$\mathbf A^{\textbf{loc}}$ from the Kolmogorov
bound of D\"obler and Peccati (\emph{Ann. Probab.}, 2018).  We also obtain
quantitative $L^4$ estimates for all iterated
Malliavin derivatives and, for $F$ in a Poisson chaos,
the fourth moment assumption of $F$
forces the $L^4$-integrability of its kernel.}

\newcommand{\E}{\mathbb{E}}
\newcommand{\PP}{\mathbb{P}}
\newcommand{\R}{\mathbb{R}}
\newcommand{\N}{\mathbb{N}}
\newcommand{\NN}{\mathcal{N}}
\newcommand{\ind}{\mathbf{1}}

\newcommand{\cG}{\mathcal{G}}
\newcommand{\cH}{\mathcal{H}}

\newcommand{\cZ}{\mathcal{Z}}
\newcommand{\cF}{\mathcal{F}}
\newcommand{\sZ}{\mathscr{Z}}
\newcommand{\bC}{\mathbb{C}^\eta}
\newcommand{\fin}{\mathbb{C}^{\bigstar}}
\newcommand{\AND}{\hspace{2mm}{\rm and}\hspace{2mm}}
\newcommand{\fH}{\mathfrak{H}}

\newcommand{\dK}{d_{\mathrm{Kol}}}
\newcommand{\dW}{d_{\mathrm W}}

\newcommand{\norm}[1]{\left\lVert #1\right\rVert}

\newcommand{\eps}{\varepsilon}
\newcommand{\noi}{\noindent}
\newcommand{\dl}{\delta}
\newcommand{\s}{\sigma}
\newcommand{\bul}{\bullet}
\newcommand{\wt}{\widetilde}
\newcommand{\dom}{\text{dom}}

\allowdisplaybreaks[2]

\hypersetup{
  pdftitle={A Kolmogorov fourth-moment bound on Poisson chaos  via a martingale core},
  pdfauthor={Guangqu Zheng}
}

\begin{document}

\section{Introduction}
\label{SEC1}

This paper is motivated by the problem of establishing a Kolmogorov bound 
for the fourth moment theorem on the Poisson space \cite{DP18,DVZ18} 
under the sole assumption of a finite fourth moment. 
We achieve this goal through an intrinsic martingale approximation (Theorem \ref{thm_core})
and our Kolmogorov bound (Theorem \ref{thm_kol}) is of the same order 
as the one in \cite{DP18} and the
 Wasserstein bound in \cite{DVZ18}.

\subsection{Fourth moment theorems on the Poisson space} \label{SEC_11}

The fourth moment phenomenon was first discovered in the Gaussian setting 
by David Nualart and Giovanni Peccati in their groundbreaking paper \cite{FMT}. 
Their result asserts that, for a sequence of normalized multiple Wiener-It\^o integrals 
of fixed order over a Gaussian field, 
convergence in distribution to a standard normal   
$N\sim\NN(0,1)$ is equivalent to convergence of 
the corresponding fourth moments to
$
\E[N^4]=3.
$
Since its introduction, the fourth moment theorem 
has become a widely used tool in stochastic analysis, 
especially when combined with its multivariate extension \cite{PT05}, 
for establishing central limit theorems for functionals of Gaussian fields.
Another major methodological breakthrough was achieved by Nourdin and Peccati
in \cite{NP09}, where Stein's method for normal approximation 
was combined with Malliavin calculus for the first time in this context. 
Among many other consequences, their approach yields Berry-Esseen-type bounds 
for a normalized multiple integral $F$ of the form
\begin{align}\label{nota1}
\textsf{dist}(F,N)
\leq C\sqrt{\kappa_4(F)},
\quad
\text{where
$\kappa_4(F):=\E[F^4]-3\E[F^2]^2$
and
$N\sim\NN(0,1)$,}
\end{align}
Here, $\textsf{dist}$ may denote several standard probability metrics, 
including 
the total variation distance, 
the Kolmogorov distance \eqref{def_dK}, 
the Wasserstein distance \eqref{def_dW}, 
and the Fortet-Mourier distance. 
The resulting Malliavin-Stein method was subsequently developed 
systematically in the monograph \cite{bluebook}.
See   \cite[Ch.~1]{Zheng18} for a brief historical account.

We now turn to the Poisson setting, 
where the corresponding problems are substantially more delicate. 
In this framework, the Malliavin derivative $D^+$ is essentially 
an add-one difference operator rather than a derivation, 
and consequently no diffusion-type chain rule is available.
The first systematic Poisson Malliavin--Stein
estimates were obtained in \cite{PSTU10,PZ10}.
The search for a general fourth-moment theorem on
the Poisson chaos remained unresolved until
D\"obler and Peccati \cite{DP18} settled the
problem by combining the discrete
Malliavin--Stein method, Mecke-type identities, and
Ledoux's spectral approach \cite{Led12}.
We collect  some important
earlier results in the following.

\begin{remark}\label{rem_prior_FMT}
Before the general result of D\"obler and Peccati
\cite{DP18}, fourth-moment theorems on the Poisson
space had been established in several important but
more restrictive settings.  For double Poisson
integrals, qualitative normal convergence criteria
in terms of kernel contractions were obtained in
\cite{PT08}.  Related results for the second and
third Poisson chaoses, under the regularity condition
\eqref{PB}, were proved in \cite{BP16}.  For multiple
integrals of arbitrary fixed order with nonnegative,
or more generally constant-sign, kernels,
quantitative normal approximation bounds were
developed in \cite{LRP13,ET14,Sch16}; in these
settings, the relevant contraction conditions can be
controlled by the fourth cumulant.  Finally, for
finite homogeneous sums in independent Poisson
random variables with intensity parameters bounded
away from zero, a universality principle and a
fourth-moment theorem were established in
\cite{PZ14}.
\end{remark}

Throughout this paper, we fix an arbitrary $\s$-finite measure space 
$(\cZ, \sZ, \mu)$ and let $\eta$ denote a Poisson random measure 
on $\cZ$ with control measure $\mu$, defined on a suitable probability space 
$(\Omega, \cF, \PP)$. 
Let $\bC_q$ denote the $q$-th Poisson chaos associated to $\eta$,
$D^+$ denote the add-one cost operator,  $L^2_s(\mu^q)$
denote the space of symmetric functions in $L^2(\cZ^q, \mu^q)$
and 
$I_q(f)$  denote a generic element of $\bC_q$.
We write $\| \bul\|_r$ for the $L^r(\Omega)$-norm
and let $\dW, \dK$ stand for the Wasserstein, Kolmogorov distances
respectively:

\noi
\begin{align}
\dW(X, Y) : &=      \sup_{\| h'\|_\infty \leq 1} |\E [ h(X)]  -\E[  h(Y) ] | 
\label{def_dW} \\
 \dK(X,Y) : &= \sup_{t\in\R} | \PP(X\leq t) - \PP(Y \leq t) |.
\label{def_dK} 
 \end{align}

\noi
 See Section \ref{SEC_21}
for further notations and background.

Now let us state the main results in \cite{DP18}.

\enlargethispage{2\baselineskip}

\begin{theorem}[\textsf{D\"obler--Peccati} {\cite[Theorem 1.3]{DP18}}] \label{thm_DP18}
Fix an integer $q\geq 1$ and consider  $F\in\bC_q$ with $\E[ F^2] =1$.
Recall the notation \eqref{nota1}.
Then, the following statements hold. 

\smallskip
\noi
{\rm (i)} Assume  $F$ satisfies Assumption $\mathbf{A}$
{\rm$\big($}i.e.,   $F\in L^4(\Omega)$ and
all of the random functions
$D^+F$, $FD^+F$, $(D^+F)^4$, and $F^3D^+F$
belong to  $L^1(\Omega\times\cZ,\PP\otimes\mu)${\rm$\big)$}.
Then,

\noi
\begin{align}\label{bdd_DP1}
\dW(F, N) \leq \Big( \sqrt{\frac{2}{\pi}} \frac{2q-1}{2q} + \frac{\sqrt{4q-3}}{\sqrt{q}} \Big)
\sqrt{\kappa_4(F)} \quad \text{with   $\kappa_4(F)\geq 0$.}
\end{align}

\noi
{\rm (ii)} Assume additionally that $F$ satisfies Assumption $\mathbf{A}^{\textbf{loc}}$ 
{\rm$\big($}i.e., there exists a set $Z_0\in\sZ$ with $\mu(\cZ\setminus Z_0) =0$
such that $D^+_zF$ verifies  Assumption  $\mathbf{A}$
for every fixed $z\in Z_0${\rm$\big)$}.
Then, 

\noi
\begin{align}\label{bdd_DP2}
\dK(F, N) 
&\leq  \min\big\{15.6,     \,\,  11 + 2^{\frac32} \big(   \|F\|_4^2 + \| F\|_4   \big)    \big\}
\sqrt{\kappa_4(F)} \\
&= 15.6 \sqrt{\kappa_4(F)}. \notag
\end{align}

\end{theorem}

Observe that the multiplicative factor in \eqref{bdd_DP2} reduces to
$15.6$. Indeed, since
$\|F\|_4\geq \|F\|_2=1$, we have
$
11+2^{\frac32}
\bigl(\|F\|_4^2+\|F\|_4\bigr)
\geq
11+2^{\frac52}
\approx 16.6569.$
%
%
Moreover, the case $\kappa_4(F) =0$ cannot occur
by
\cite[Prop. 1.6]{DP18}; see  
\cite[Rem. 5.2]{DVZ18} and \eqref{intro_kernel} below.

D\"obler and Peccati themselves pointed out that Assumptions
$\mathbf A$ and  
$\mathbf A^{\textbf{loc}}$ appear to be artifacts of the techniques
employed in \cite{DP18}, rather than intrinsic conditions for the
fourth moment phenomenon. More precisely, these assumptions were
used to justify the relevant Mecke-type identities and the
almost-sure representations of the Malliavin and
carr\'e-du-champ operators appearing in their proofs. 
Nevertheless, Theorem \ref{thm_DP18} provided the first general fourth moment bound in
Wasserstein and Kolmogorov distances for multiple   integrals on the Poisson space.
As observed in \cite[Rem. 1.2]{DP18}, Assumptions
$\mathbf A$ and $\mathbf A^{\textbf{loc}}$ are satisfied whenever

\noi
  \begin{align} \label{PB}
  \begin{rcases}
  &\text{${\displaystyle F= \sum_{\ell\leq M} I_\ell(f_\ell)}$, where 
  $M <\infty$ and each kernel  $f_\ell\in L^2_s(\mu^\ell)$ is  bounded}\\[-0.3em]
 &\text{with support   contained in $A^\ell$
  for some $A\in\sZ$ with $\mu(A) <\infty$.}
  \end{rcases}
 \end{align}
The condition \eqref{PB} already appeared in the work \cite{BP16}
by Bourguin and Peccati.
Soon after the appearance of the work \cite{DP18}, 
D\"obler, Vidotto, and the author \cite{DVZ18} 
  used an infinitesimal exchangeable-pairs construction
 to establish univariate and multivariate
fourth-moment bounds under the sole assumption of finite fourth
moments. 
They obtained the following Wasserstein bound  
 with a slightly better multiplicative factor.

\begin{theorem}[\textsf{D\"obler--Vidotto--Zheng} {\cite[Theorem 1.2 and 
Corollary 1.3]{DVZ18}}] \label{thm_DVZ18}
Fix an integer $q\geq 1$ and consider  $F\in\bC_q\cap L^4(\Omega)$ with $\E[ F^2] =1$.
Then,  with the notation \eqref{nota1},

\noi
\begin{align}\label{bdd_DVZ}
\dW(F, N) 
\leq 
\Big( \sqrt{\frac{2}{\pi}} \frac{2q-1}{2q} 
+ \frac{2}{3} \frac{\sqrt{4q-3}}{\sqrt{q}} 
\Big)
\sqrt{\kappa_4(F)}.
\end{align}
In particular, if $F_n\in\bC_{q_n}$ satisfy $(\E[ F^2_n], \E[ F^4_n])  \to (1, 3)$,
then $F_n\to N$ in law.
\end{theorem}

The derivation of the bound \eqref{bdd_DVZ} relies crucially on
 the regularity properties of the solution to the Stein equation 
 associated with the Wasserstein distance \eqref{def_dW}. 
 Consequently, the arguments developed in \cite{DVZ18} 
 do not directly yield a Kolmogorov bound of order $\sqrt{\kappa_4(F)}$. 
In general, within Stein's method, obtaining bounds in
Kolmogorov distance is   more difficult than obtaining
Wasserstein bounds.  This is mainly because the Kolmogorov distance
involves discontinuous indicator test functions, whereas the
Wasserstein distance is based on Lipschitz test functions and therefore
benefits from better regularity of the associated Stein solutions; see,
e.g., \cite{CGS11}.  As a result, Kolmogorov bounds often
require additional smoothing or concentration arguments and may incur
a loss in the convergence rate.
As already pointed out in \cite[Rem.~1.4]{DVZ18}, 
 the standard smoothing argument leads only to a rate of order $\kappa_4(F)^{1/4}$.


\subsection{Main result: \textsf{a martingale core}} \label{SEC_12}

As already mentioned below Theorem \ref{thm_DP18},
if a random variable $F\in\bC_q$ satisfies the condition \eqref{PB},
Assumptions $\mathbf{A}$ and $\mathbf{A}^{\textbf{loc}}$
hold. In this regular case, the Kolmogorov bound \eqref{bdd_DP2} holds.
It is then very natural to seek  a regular approximation
 that converges in $L^4(\Omega)$
 so that relevant fourth moment estimates can be passed along 
 this regular approximation. 
Martingale convergence provides precisely this topology.  
However, an arbitrary Doob martingale does not preserve the Poisson chaos
decomposition. 

\begin{example} \label{ex1}
\rm
Let $\eta$ be a Poisson random
measure on $[0,1]$ with Lebesgue control measure, and consider
$
F
:=
\eta([0,1])-1
\in\bC_1.$
For $n\geq1$, let
$
A_{n,k}
:=
\big(   \frac{k-1}{2^n}, \frac{k}{2^n}  \bigr],
$
$1\leq k\leq 2^n,
$
and introduce the occupancy variables and a filtration:
\[
\xi_{n,k}
:=
\ind_{\{\eta(A_{n,k})\geq1\}}
\quad
{\rm and}
\quad
\cH_n
:=
\s\big\{
\xi_{n,k}:1\leq k\leq2^n
\big\}, \,\,
n\geq 1.
\]
Since the dyadic partitions are nested, $(\cH_n)_{n\geq1}$ is an
increasing filtration.  Moreover, the Poisson process is almost surely
simple (\cite[Proposition 6.9]{LP18}), and hence
$\s\big( 
\bigcup_{n\geq1}\cH_n \big)
=
\s\{\eta\}$
up to null sets.
Consequently, we have the following martingale approximation:
for every finite $p \geq 1$,
\[
F_n
:=
\E[F | \cH_n]
\xrightarrow[\text{in $L^p(\Omega)$}]{n\to+\infty} F.
\]

\noi
However, with $\eta([0,1]) =  \sum_{k=1}^{2^n} \eta(A_{n,k})$,
we can write

\noi
\begin{align*}  
\begin{aligned} 
F_n
= -1 &+   \sum_{k=1}^{2^n} 
 \E\big[ \eta(A_{n,k} ) |   \eta(A_{n,k} ) \geq 1 \big] \ind_{\{  \eta(A_{n,k} ) \geq 1  \} }  
 =\frac{a_n}{p_n}
\sum_{k=1}^{2^n}
\bigl(\xi_{n,k}-p_n\bigr),
\end{aligned}
\end{align*}
where
$
a_n:=2^{-n}$,
$p_n:=1-e^{-a_n}$,
and we used the elementary fact that $\E[ X | X\geq 1] = \frac{\lambda}{1 - e^{-\lambda}}$
for $X\sim\text{Poisson}(\lambda)$.
For $z\in A_{n,k}$ and
$z_1,z_2\in A_{n,k}$, one has
\[
D_z^+F_n
= \frac{a_n}{p_n} \ind_{ \{  \eta(A_{n,k}) =0 \}}
\AND
 D^+_{z_1}D^+_{z_2}F_n 
=
- \frac{a_n}{p_n} \ind_{ \{  \eta(A_{n,k}) =0 \}}.
\]
Hence,
Stroock's formula \eqref{for1}   shows that 
 $F_n$ does not
belong to a single Poisson chaos.

\end{example} 

Here is another  naive  martingale approximation by  spatial truncations. 

\begin{example}  \label{ex2} \rm
 Let
$(Z_n)_{n\geq1}$ be an increasing sequence in $\sZ$ such that
$\mu(Z_n)<\infty$
and
$Z_n\uparrow \cZ$.
Set
$
\mathcal F_n^{\mathrm{sp}}
:=
\sigma\big\{
\eta(B):B\in\sZ,\ B\subseteq Z_n
\big\}.
$
For $F=I_q(f)\in L^4(\Omega)$, independence of the restrictions of
the Poisson measure gives
$
\wt{F}_n
:=
\E[F | \cF_n^{\mathrm{sp}}]
=
I_q\bigl(f\ind_{Z_n^q}\bigr).
$
Spatial conditioning does preserve the chaos order,
but it does not regularize an unbounded kernel. 
Direct step-kernel approximation has the opposite defect: it gives regularity and 
$L^2$-convergence, but not the $L^4$-convergence supplied by a martingale.

\end{example}

The problem is
therefore to construct a filtration whose conditional expectations
simultaneously preserve the chaos order, regularize the kernels,
and converge in $L^4(\Omega)$.  
This is precisely the purpose of the
finite-count martingale core introduced below.

\begin{definition}[\textsf{Finite-partition core}]
\label{def_core}
For an integer $q\geq1$, let $\fin_q$ consist of all $F\in\bC_q$
with the property \eqref{PB}.
 We also put $\fin_0 = \bC_0$ as a convention.
\end{definition}

In view of the product formula \cite{Surg84, DPF18},
every element of $\fin_q$ has moments of all orders and
satisfies Assumptions $\mathbf A$ and $\mathbf{A}^{\textbf{loc}}$ of
\cite{DP18}.  The next result is the main approximation theorem of the
paper.

\begin{theorem}[\textsf{Simultaneous fixed-chaos martingale core}]
\label{thm_core}
Fix $p\in[2,\infty)$ and let
\[
  F_j=I_{q_j}(f_j)\in L^p(\Omega),
  \qquad j=1,\ldots,d,
\]
where $d\geq1$, $q_j\geq1$, and
$f_j\in L_s^2(\mu^{q_j})$.  There exist increasing finite-dimensional
subspaces $H_n\subset L^2(\mu)$, their orthogonal projections $P_n$,
and increasing $\s$-algebras
\begin{align}
  \cG_n
  =\sigma\big\{ \eta(A_{n,1}),\ldots,\eta(A_{n,m_n})\big\},
  \label{Gn_def}
\end{align}
where the sets $\{A_{n,k}: 1\leq k \leq m_n\}$ are pairwise disjoint 
with $\mu( A_{n,k})\in(0, \infty)$
such that
\begin{align}
  F_{j,n}
  :=\E[F_j|\cG_n]
  =I_{q_j}\bigl(P_n^{\otimes q_j}f_j\bigr)
  \in\fin_{q_j}.
  \label{core_id}
\end{align}
Moreover, for any $j\in\{1, ...,d\}$,
we have
\begin{align}
  F_{j,n}\xrightarrow[\text{in }L^p(\Omega)]{n\to+\infty} F_j.
  \label{core_Lp}
\end{align}

\end{theorem}

In Example \ref{ex1}, 
the occupancy-based approximation mixes different Poisson chaoses, 
whereas in Example \ref{ex2}, the naive spatial truncation fails to 
satisfy the regularity condition \eqref{PB}. 
By contrast, Theorem \ref{thm_core} shows that 
conditioning on exact cell counts both preserves the chaos order 
and yields the desired regularity.
We present the proof of Theorem \ref{thm_core}
in Section \ref{SEC_22}.

\subsection{Applications} \label{SEC_13}
 Now let us present a few consequences of Theorem \ref{thm_core}.
The first application of the martingale core is a hierarchy
of $L^4$ estimates for iterated Malliavin derivatives.

\begin{theorem}\label{app1}
Let  $F\in\bC_q\cap L^4(\Omega)$ for some $q\geq 1$,
and let
$D^rF$ denote the $r$-th 
add-one cost derivative as in \eqref{def_MD}.
Then, we have $D^r F\in L^4(\Omega\times \cZ^r)$ for every $1 \leq r \leq q$
with

\noi
\begin{align}
  \int_{\cZ^r}\E\bigl[|D^r_{z_1, ..., z_r}F|^4\bigr]
  \,\mu(dz_1)\cdots \mu(dz_r)
  \leq a_{q,r}\kappa_4(F),
  \label{intro_der}
\end{align}
where 
$a_{q,r}
  :=\prod_{j=0}^{r-1}\big(4(q-j)-3\big).$
In particular, 

\noi
\begin{align}
  (q!)^4\norm{f}_{L^4(\mu^q)}^4
  \leq
  \Big(\prod_{j=1}^q(4j-3)\Big)\kappa_4(F).
  \label{intro_kernel}
\end{align}
\end{theorem}

Thus, fourth integrability of a Poisson multiple integral forces fourth
integrability of its kernel.  The regular core is in fact dense for
the graph norm containing the $L^4$ norms of all derivatives up to
order $q$; see Corollary \ref{cor_graph}.

The second application of Theorem \ref{thm_core}
is the anticipated Kolmogorov bound for the fourth moment theorem
under the sole assumption of finite fourth moment.

\begin{theorem}[\textsf{Kolmogorov bound}]
\label{thm_kol}
Fix an integer $q\geq 1$ and consider  $F\in\bC_q\cap L^4(\Omega)$ with $\E[ F^2] =1$.
Then,  with $N\sim\NN(0,1)$,
$
\dK(F, N) 
\leq  15.6
\sqrt{\E[ F^4] - 3}.$
 
\end{theorem}

Theorem~\ref{thm_kol} removes Assumptions $\mathbf A$ and
$\mathbf{A}^{\textbf{loc}}$ from \eqref{bdd_DP2}.
The proof is given in Section \ref{SEC_24}.
We also extend \cite[Lemma 3.2]{DP18} under the sole
assumption $F\in L^4(\Omega)\cap\bC_q$, although this
extension is not needed for Theorem \ref{thm_kol}.
The finite-chaos inclusion for $FG$ also follows from
\cite[Corollary~5.3]{CP25} under the weaker assumption
$FG\in L^2(\Omega)$.
The additional point here is the constructive spectral
closure: a common finite-count filtration yields
simultaneous approximation of $(F, G)$, their product,
and the associated carr\'e-du-champ.

\begin{proposition}[\textsf{Graph closure of the fixed-chaos estimates}]
\label{prop_DP}
Let $p, q\geq1$, let $F\in L^4(\Omega)\cap \bC_p$ and $G\in L^4(\Omega)\cap \bC_q$.
Let $F_n, G_n$ be the
 martingale approximants from
Theorem~\ref{thm_core}.  Then
 $ FG\in\bigoplus_{r=0}^{p+q}\bC_r.$
Let $\Gamma(F, G)$ denote the carr\'e-du-champ operator \eqref{Gamma_def}.
Then, we have
\begin{align}
  \Gamma(F_n, G_n)&\xrightarrow[\text{in }L^2(\Omega)]{n\to+\infty}\Gamma(F,G).
  \label{gam_cvg}
\end{align}
Moreover, the following  exact identity holds:
\begin{align}
  \frac1{2p}
  \int_\cZ\E[|D_zF|^4] \,\mu(dz)
  =\frac{3}{p}\E[F^2\Gamma(F,F)]-\E[F^4].
  \label{D4_id}
\end{align}
As a result,  the  estimates in
\cite[Lemma 3.2]{DP18} are valid under 
$F\in L^4(\Omega)\cap\bC_p$.
\end{proposition}

\section{Proofs}
\label{SEC2}

\subsection{Basic stochastic analysis on the Poisson space}
\label{SEC_21}

Recall that $\eta$ is a proper Poisson random measure on
$(\cZ,\sZ)$ with $\s$-finite control measure $\mu$; see
\cite[Corollary~3.7]{LP18}.  For $q\geq1$, let $I_q$ denote
the $q$th multiple Wiener--It\^o integral with respect to the
compensated measure $\eta-\mu$, and let $L_s^2(\mu^q)$ be the
symmetric subspace of $L^2(\mu^q)$.  We use the conventions
$I_0(c)=c$ and $L_s^2(\mu^0)=\R$.  The $q$th Poisson chaos is
\begin{align}
\bC_q
:=
\{I_q(f):f\in L_s^2(\mu^q)\},
\qquad
\bC_0=\R.
\label{def_CP}
\end{align}
Every $F\in L^2(\Omega,\s\{\eta\},\PP)$ admits the unique
Wiener--It\^o--Poisson expansion
\begin{align}
F
=
\E[F]
+
\sum_{q\geq1}I_q(f_q),
\qquad
f_q\in L_s^2(\mu^q),
\label{decomp}
\end{align}
with convergence in $L^2(\Omega)$; see
\cite[Theorem~18.10]{LP18}.

If $F=\mathfrak f(\eta)$, its add-one cost is
\[
D_z^+F
:=
\mathfrak f(\eta+\dl_z)-\mathfrak f(\eta).
\]
Higher-order differences are defined recursively by
$D^{(1)}=D^+$ and
\begin{align}
D^{(r)}_{z_1,\ldots,z_r}F
:=
D_{z_1}^+
\bigl(D^{(r-1)}_{z_2,\ldots,z_r}F\bigr),
\qquad r\geq2.
\label{def_MD}
\end{align}
For $F\in L^2(\Omega,\s\{\eta\},\PP)$, Stroock's formula gives
\begin{align}
f_q(z_1,\ldots,z_q)
=
\frac{1}{q!}
\E\bigl[D^{(q)}_{z_1,\ldots,z_q}F\bigr]
\label{for1}
\end{align}
for $\mu^q$-almost every
$(z_1,\ldots,z_q)\in\cZ^q$.

The modified Wiener--It\^o isometry reads
\begin{align}
\E[I_p(f)I_q(g)]
=
\ind_{\{p=q\}}q!
\langle f,g\rangle_{L^2(\mu^q)}.
\label{iso_mod}
\end{align}
Consequently, the map
\begin{align}
(f_q)_{q\geq0}
\longmapsto
\sum_{q\geq0}I_q(f_q)
\label{fock_iso}
\end{align}
is unitary from the symmetric Poisson Fock space
\[
\mathfrak F_\mu
:=
\left\{
(f_q)_{q\geq0}:
\sum_{q\geq0}q!\norm{f_q}_{L^2(\mu^q)}^2<\infty
\right\}
\]
onto $L^2(\Omega,\s\{\eta\},\PP)$; see
\cite[Chapter~18]{LP18} and \cite{LP11,Last16}.

For $F$ with expansion \eqref{decomp}, set
\[
\dom(D)
:=
\left\{
F:
\sum_{q\geq1}q q!
\norm{f_q}_{L^2(\mu^q)}^2<\infty
\right\}.
\]
For $F\in\dom(D)$,
\[
D_zF
:=
\sum_{q\geq1}qI_{q-1}(f_q(z,\cdot)),
\]
with convergence in
$L^2(\Omega\times\cZ,\PP\otimes\mu)$, and
$D_zF=D_z^+F$ for $\PP\otimes\mu$-almost every
$(\omega,z)$; see \cite[(19)]{DP18}.  The
Ornstein--Uhlenbeck generator is defined on
\[
\dom(L)
:=
\left\{
F:
\sum_{q\geq1}q^2q!
\norm{f_q}_{L^2(\mu^q)}^2<\infty
\right\}
\]
by
\[
LF
:=
-\sum_{q\geq1}qI_q(f_q).
\]
Finally, whenever $F,G,$ and $FG$ belong to $\dom(L)$, define
\begin{align}
\Gamma(F,G)
:=
\frac{1}{2}
\bigl(L(FG)-F\,LG-G\,LF\bigr).
\label{Gamma_def}
\end{align}
If $F\in\bC_p$ and $G\in\bC_q$, then
\begin{align}
\Gamma(F,G)
=
\frac{1}{2}
\bigl(L(FG)+(p+q)FG\bigr).
\label{Gamma_chaos}
\end{align}
We refer to \cite{Last16}, \cite[Chapter~20]{LP18}, and
\cite[Section~2]{DP18} for further details.

\subsection{Construction of the martingale core} \label{SEC_22}

We are now ready to prove Theorem \ref{thm_core}.

\begin{proof}[Proof of Theorem \ref{thm_core}]
For each $j$, finite linear combinations of indicators of measurable
rectangles
\[
B_1\times\cdots\times B_{q_j}
\quad\text{with}
\quad
\mu(B_i)<\infty,
\]
are dense in $L^2(\mu^{q_j})$.  Since symmetrization is an
$L^2$-contraction and $f_j$ is symmetric, we may therefore choose
symmetric rectangular step functions $u_{j,n}$ such that
\noi
\begin{align}
\norm{u_{j,n}-f_j}_{L^2(\mu^{q_j})}
\leq  1/n.
\label{u_approx}
\end{align}

Let $\mathcal B_n$ denote the finite collection of all coordinate
sets appearing in the rectangular representations of $u_{j,k}$,
for $1\leq j\leq d$ and $1\leq k\leq n$, and set
\[
C_n
:=
\bigcup_{B\in\mathcal B_n}B.
\]
It is clear that $\mu(C_n)<\infty$.  Let
\[
\mathcal P_n
=
\{A_{n,1},\ldots,A_{n,m_n}\}
\]
be the collection of non-null atoms of the finite algebra generated
by $\mathcal B_n$ inside $C_n$, where complements are taken relative
to $C_n$.  Thus, the sets in $\mathcal P_n$ are pairwise disjoint,
have finite positive measure, and form a partition of $C_n$ modulo
a $\mu$-null set.  Moreover, every $B\in\mathcal B_n$ is, up to a
null set, a union of cells from $\mathcal P_n$.

Since $\mathcal B_n\subseteq\mathcal B_{n+1}$,
one has

\noi
\begin{align*}
C_n
\subseteq C_{n+1}
\AND
\mathcal P_{n+1}
\text{ refines }\mathcal P_n
\text{ on }C_n,
\end{align*}
where the refinement is understood modulo $\mu$-null sets.

Set
\[
\fH_n
:=
\operatorname{span}
\{
\ind_{A_{n,1}},\ldots,
\ind_{A_{n,m_n}}
\}
\subset L^2(\mu),
\]
and let $P_n$ denote the orthogonal projection from $L^2(\mu)$ onto
$\fH_n$.  The refinement property implies that
$
\fH_n\subseteq \fH_{n+1}.
$
Furthermore, since every coordinate set occurring in $u_{j,n}$ is
a union of cells from $\mathcal P_n$, we have
$
u_{j,n}\in \fH_n^{\odot q_j},
$
where  $\fH_n^{\odot q_j}$ denotes the $q_j$-th  symmetric tensor power of $\fH_n$.
Now $P_n^{\otimes q_j}$ is the orthogonal projection from
$L^2(\mu^{q_j})$ onto $\fH_n^{\otimes q_j}$.  Since
$u_{j,n}\in \fH_n^{\odot q_j}\subseteq \fH_n^{\otimes q_j}$, the
best-approximation property of orthogonal projections, with \eqref{u_approx},
 yields
\noi
\begin{align}
\| P_n^{\otimes q_j} f_j-f_j\|_{L^2(\mu^{q_j})}
&\leq \| u_{j,n}-f_j \|_{L^2(\mu^{q_j})}
\leq
1/n.
\label{Pn_approx}
\end{align}

More explicitly, $P_n^{\otimes q_j}f_j$ has the representation
\noi
\begin{align*}
P_n^{\otimes q_j}f_j(z_1, ..., z_{q_j})
=
\sum_{r_1,\ldots,r_{q_j}=1}^{m_n}
c_{j,n}(r_1,\ldots,r_{q_j})
\prod_{\ell=1}^{q_j}
\ind_{A_{n,r_\ell}}(z_\ell)
\end{align*}

\noi
for suitable real coefficients
$c_{j,n}(r_1,\ldots,r_{q_j})$.  Consequently,
$P_n^{\otimes q_j}f_j$ is symmetric and bounded, is constant on
each product cell of $\mathcal P_n^{q_j}$, and vanishes outside
$C_n^{q_j}$.

Define $\cG_n$ by \eqref{Gn_def}.  Since every cell of
$\mathcal P_n$ is, modulo a null set, a union of cells from
$\mathcal P_{n+1}$, each count $\eta(A_{n,k})$ is a finite sum of
level-$(n+1)$ counts.  Consequently,
$
\cG_n\subseteq\cG_{n+1}.
$

For $1\leq k\leq m_n$, set
\[
N_{n,k}:=\eta(A_{n,k}),
\qquad
\lambda_{n,k}:=\mu(A_{n,k}).
\]
The random variables $N_{n,1},\ldots,N_{n,m_n}$ are independent,
with $N_{n,k}\sim\text{Poisson}(\lambda_{n,k})$.  For a
multi-index
$\alpha=(\alpha_1,\ldots,\alpha_{m_n})\in\N_0^{m_n}$ (with $\N_0=\{0,1, ... \}$), 
define
\[
C_\alpha(N_n)
:=
\prod_{k=1}^{m_n}
C_{\alpha_k}(N_{n,k};\lambda_{n,k})
\AND
|\alpha|
:=
\sum_{k=1}^{m_n}\alpha_k,
\]
where $C_\ell(\,\cdot\,;\lambda)$ denotes the  monic  Charlier
polynomial of degree $\ell$
(see \cite[Def.~6.2.7]{Pri09}). 
  The tensor-product Charlier
polynomials form a complete orthogonal system in $L^2(\cG_n)$
with
\[
C_\alpha(N_n)
=
I_{|\alpha|} \Big(\wt{\bigotimes_{k=1}^{m_n} \ind_{A_{n,k}}^{\otimes\alpha_k}} \Big).
\]
See, e.g., 
\cite[Prop.~6.2.9 and
Lem.~6.2.10]{Pri09}. 
The kernels on the right-hand side, with $|\alpha|=r$, span
$\fH_n^{\odot r}$.  Grouping the complete Charlier system according
to total degree therefore gives the orthogonal decomposition
\begin{align}
L^2(\cG_n)
=
\bigoplus_{r=0}^{\infty}
I_r(\fH_n^{\odot r}).
\label{finite_Fock}
\end{align}

\noi
Equivalently, under the Poisson Fock-space isometry \eqref{fock_iso}, the
one-particle projection $P_n$ has multiplicative second quantization
$
  \textsf{Exp}(P_n)
  :=\bigoplus_{r=0}^{\infty}P_n^{\otimes r}.
$
Its range is the Fock space over $\fH_n$, which is exactly the space in
\eqref{finite_Fock};
see, e.g., \cite[Section 5]{Surg84}.  Hence, $\textsf{Exp}(P_n)$ is the
orthogonal projection onto $L^2(\cG_n)$ and therefore
\begin{align}
  \textsf{Exp}(P_n)[X]
  =\E[X|\cG_n]
\quad  \text{for}\quad X\in L^2(\Omega,\s\{\eta\}, \PP).
  \label{Exp_cond}
\end{align}
 Restricting \eqref{Exp_cond} to the
$q_j$-th chaos gives \eqref{core_id}.

The martingale convergence theorem yields
\[
  F_{j,n}
  \xrightarrow[\text{in }L^p(\Omega)]{n\to+\infty} \E[F_j|\cG_\infty]
 \quad
 {\rm with}
 \quad
  \cG_\infty:=\sigma\Big(\bigcup_{n\geq1}\cG_n\Big).
\]
On the other hand, we deduce from \eqref{Pn_approx} that
\[
  \E[(F_{j,n}-F_j)^2]
  =q_j!\| P_n^{\otimes q_j}f_j-f_j \|_{L^2(\mu^{q_j})}^2
  \xrightarrow{n\to+\infty}0,
\]
which ensures  
$\E[F_j|\cG_\infty]=F_j$ almost surely.  This proves
\eqref{core_Lp} for each $j\in\{1, ..., d\}$.
\end{proof}

\subsection{\texorpdfstring{$L^4$-Malliavin regularity}{L4-Malliavin regularity}}
\label{SEC_23}

For $F=I_q(f)$ with $f\in L^2_s(\mu^q)$, the iterated Malliavin derivative is
\begin{align}
  D^r_{z_1,\ldots,z_r}F
  =(q)_r \, I_{q-r}\bigl(f(z_1,\ldots,z_r,\cdot)\bigr)
  \quad\text{with   $(q)_r:=\frac{q!}{(q-r)!}$}
  \label{iter_D}
\end{align}
for $\mu^r$-almost every $(z_1,\ldots,z_r)\in\cZ^r$,  $1\leq r\leq q$.
In particular,
$D^qF=q!f$ is deterministic.

\begin{proof}[Proof of Theorem \ref{app1}]
Let $F_n=I_q(f_n)$ be the approximations from
Theorem~\ref{thm_core} with $p=4$ such that $F_n\in\fin_q$ 
as in Definition~\ref{def_core}.  
We recall   from \cite[Lem.~3.2]{DP18}
that for $G\in\fin_q$ (so that $G$ satisfies   Assumption $\mathbf{A}$ therein)

\noi
\begin{align}
  \int_\cZ\E[|D_zG|^4]\,\mu(dz)
  &\leq(4q-3)\kappa_4(G).
  \label{reg_D4}
\end{align}

\noi
Applying
\eqref{reg_D4} to $F_n$ gives
\begin{align}
   \int_\cZ\E[|D_z F_n|^4]\,\mu(dz)
  \leq(4q-3)\kappa_4(F_n).
  \label{D_base_n}
\end{align}
Fix $2\leq r\leq q$.  For $\mu^{r-1}$-almost every
$(z_1,\ldots,z_{r-1})\in\cZ^{r-1}$, the random variable
\[
  D^{r-1}_{z_1, ..., z_{r-1}}F_n \in\fin_{q-r+1},
\]
since its kernel is still bounded, stepwise constant, and supported in
a finite-measure rectangle. 
 Therefore, we have 
\[
\begin{aligned}
  \int_\cZ\E\bigl[|D_y    D^{r-1}_{z_1, ..., z_{r-1}}F_n     |^4\bigr] \,\mu(dy)
  &\leq
  \bigl(4(q-r+1)-3\bigr)
  \kappa_4( D^{r-1}_{z_1, ..., z_{r-1}}F_n)\\
  &\leq
  \bigl(4(q-r+1)-3\bigr)
  \E\bigl[| D^{r-1}_{z_1, ..., z_{r-1}}F_n |^4\bigr].
\end{aligned}
\]
Integrating in $z_1, ..., z_{r-1}$ gives the recursion
\begin{align}
  M_r(F_n):&= \int_{\cZ^r}\E[|D^r_{z_1, ..., z_r} F_n|^4]\,\mu(dz_1)\ldots \mu(dz_r)
   \label{MRN} \\
  &\leq
  \bigl(4(q-r+1)-3\bigr) \int_{\cZ^{r-1}}\E[|D^{r-1}_{z_1, ..., z_{r-1}} F_n|^4]\,\mu(dz_1)\ldots \mu(dz_{r-1}).
  \notag
\end{align}
Combining the above recursion with  \eqref{D_base_n}  yields
\begin{align}
  M_r(F_n)
  \leq a_{q,r} \, \kappa_4(F_n).
  \label{D_n_bdd}
\end{align}

Since $F_n\to F$ in $L^4(\Omega)$,
\begin{align}
  \kappa_4(F_n)\longrightarrow\kappa_4(F).
  \label{kappa_cvg}
\end{align}
The  relations \eqref{iso_mod} 
and \eqref{iter_D} imply
$\E\big[ \| D^r(F_n - F) \|^2_{L^2(\mu^r)} \big] 
=(q)_r \, \E[(F_n-F)^2] \to 0$,
so that 
we can  select a subsequence along which
$D^rF_n\to D^rF$ almost everywhere on $\Omega\times \cZ^r$.  
Hence, the   bound \eqref{intro_der}
follows from 
Fatou's lemma, \eqref{D_n_bdd}, and \eqref{kappa_cvg}.
\qedhere

\end{proof}

The approximation is therefore much stronger than $L^4(\Omega)$ convergence
of the random variables alone.

\begin{corollary}[\textsf{Fourth-order graph-norm density}]
\label{cor_graph}
In the setting of Theorem~\ref{app1}, the approximants
$F_n$ can be chosen such that 

\noi
\begin{align}
  \E[|F_n-F|^4]
  +\sum_{r=1}^q
 \E\bigl[   \| D^r(F_n-F)\|^4_{L^4(\mu^r)}\bigr]
  \longrightarrow0.
  \label{graph_cvg}
\end{align}
\end{corollary}

\begin{proof}
Set $R_n:=F_n-F$.  Then $R_n\in\bC_q$ and
$R_n\to0$ in $L^4(\Omega)$. Using  \eqref{MRN},
we deduce from 
  Theorem~\ref{app1} that 
$
  M_r(R_n)
  \leq a_{q,r}\kappa_4(R_n)
  \leq a_{q,r}\E[|R_n|^4],
$
which proves \eqref{graph_cvg}.
\end{proof}

 \subsection{Kolmogorov bound for fourth moment theorem} \label{SEC_24}

 Let us first record an elementary lemma.

\begin{lemma}[\textsf{Lower semicontinuity of Kolmogorov distance}]
\label{lem_lsc}
Let $X_n$, $X$, and $Y$ be real-valued random variables.  If
$X_n\to X$ in distribution and the distribution function of $Y$ is
continuous, then
$\dK(X,Y)
  \leq\liminf_{n\to\infty}\dK(X_n,Y).
$
\end{lemma}

\begin{proof}
Choose a subsequence $(n_k)$ along which
$
\dK(X_{n_k},Y)
\to 
\liminf_{n\to\infty} \dK(X_n,Y).
$
By   Skorohod representation theorem, this subsequence and $X$ may
be realized on a common probability space so that
$X_{n_k}\to X$ almost surely.  Let $F_Y$ be the distribution function
of $Y$ and define
$\omega_Y(\eps)
  :=\sup\{   (F_Y(x+\eps)-F_Y(x-\eps)) :x\in\R \}.
  $
A continuous distribution function is uniformly continuous on $\R$,
because it has finite limits at $-\infty$ and $+\infty$.  Hence
$\omega_Y(\eps)\to0$ as $\eps\downarrow0$.

For every $x\in\R$ and $\eps>0$,
$
  \PP(X\leq x)
  \leq\PP(X_{n_k}\leq x+\eps)
  +\PP(|X_{n_k}-X|>\eps)
$
and
$
  \PP(X_{n_k}\leq x-\eps)
  \leq\PP(X\leq x)
  +\PP(|X_{n_k}-X|>\eps).
$
Consequently,
\[
\begin{aligned}
  \bigl|\PP(X\leq x)-\PP(Y\leq x)\bigr|
  \leq{}&\dK(X_{n_k},Y)+\omega_Y(\eps)
  +\PP(|X_{n_k}-X|>\eps).
\end{aligned}
\]
Taking the supremum over $x$, letting $k\to\infty$, and then 
$\eps\downarrow0$, concludes the proof. 
\end{proof}

 \begin{proof}[Proof of Theorem \ref{thm_kol}]
Let $F_n$ be the approximations of $F$ from
Theorem~\ref{thm_core} with $p=4$. 
Note that $\s_n := \sqrt{\E[ F_n^2]} > 0$ for sufficiently large $n$, 
 since $\s_n\to 1$.

 As $F_n\in\fin_q$, the bound \eqref{bdd_DP2}
 holds and 
 $
 \dK(F_n, \s_n N) = \dK( \frac{F_n}{\s_n},  N) \leq 15.6 \sqrt{\kappa_4(  \frac{F_n}{\s_n} )}.
 $
 Moreover,
 \[
 \kappa_4(F_n/\s_n)
 =
 \frac{\kappa_4(F_n)}{\s_n^4}
 \longrightarrow
 \kappa_4(F),
\AND
 \dK(\s_nN,N)\to0
 \]
 since $\s_n\to1$.
 It follows from the triangle inequality for $\dK$ and Theorem~\ref{thm_core}
 that 
 \[
 \liminf_{n\to+\infty} \dK(F_n, N) \leq 
  \liminf_{n\to+\infty} [\dK(F_n, \s_n N) +  \dK( \s_n N, N ) ] \leq 
 15.6 \sqrt{\kappa_4( F )}.
 \]
 Hence, the proof is completed by invoking Lemma \ref{lem_lsc}.
 \end{proof}

Finally, let us prove  Proposition \ref{prop_DP}.

\begin{proof}[Proof of Proposition \ref{prop_DP}]
 $F_n$ and $G_n$ are  homogeneous Charlier
polynomials of total degree $p$ and $q$, respectively.
Consequently,
\[
  F_nG_n\in  \bigoplus_{r=0}^{p+q}\fin_r.
\]
From the  $L^4(\Omega)$-convergence $(F_n, G_n)\to (F,G)$, 
one has $F_nG_n\to FG$ in $L^2(\Omega)$.  The  space $ \bigoplus_{r=0}^{p+q}\bC_r$
is a closed subspace of $L^2(\Omega)$, which proves 
$FG\in \bigoplus_{r=0}^{p+q}\bC_r.$

 By definition,  the Ornstein-Uhlenbeck generator $L$
is bounded on $ \bigoplus_{r=0}^{p+q}\bC_r$.  Therefore,
\[
  L(F_nG_n)\longrightarrow L(FG)
  \quad\text{in }L^2(\Omega),
\]
and \eqref{gam_cvg} follows from the spectral definition of the
carr\'e-du-champ; see \eqref{Gamma_def}--\eqref{Gamma_chaos}.

Corollary~\ref{cor_graph} gives
$DF_n\longrightarrow DF$
in   $L^4(\Omega\times \cZ),$
so that
\begin{align}
  \int_\cZ\E[|D_zF_n|^4] \,\mu(dz)
  \longrightarrow
  \int_\cZ\E[|D_zF|^4] \,\mu(dz).
  \label{D4_cvg}
\end{align}
For every $n$, \cite[Lemma~3.2]{DP18} gives
\[
  \frac1{2p}
  \int_\cZ\E[|D_zF_n|^4] \,\mu(dz)
  =\frac{3}{p}\E[F_n^2\Gamma(F_n,F_n)]-\E[F_n^4].
\]
Note that the relation  \eqref{gam_cvg} and $F_n\to F$ in $L^4(\Omega)$
imply the $L^1(\Omega)$-convergence of 
$  F_n^2\Gamma(F_n,F_n)
\to F^2\Gamma(F,F)$,
which together 
 with \eqref{D4_cvg} proves
\eqref{D4_id}.
\qedhere

\end{proof}


\begin{thebibliography}{99}



\bibitem[BP16]{BP16}
Bourguin, S. and Peccati, G.:
The Malliavin--Stein method on the Poisson space.
In: G. Peccati and M. Reitzner (eds.),
\emph{Stochastic Analysis for Poisson Point Processes},
Bocconi \& Springer Series, vol.~7, pp.~185--228,
Springer, Cham, 2016.

\bibitem[CGS11]{CGS11}
Chen, L.H.Y., Goldstein, L. and Shao, Q.-M.:
\emph{Normal Approximation by Stein's Method}.
Probability and Its Applications,
Springer, Heidelberg, 2011.


\bibitem[CP25]{CP25}
Cristofaro, L. and Peccati, G.:
The Poisson multiplication formula.
arXiv:2505.11389.

\bibitem[DP18a]{DP18}
D\"obler, C. and Peccati, G.:
The fourth moment theorem on the Poisson space.
\emph{Ann. Probab.} \textbf{46}, no.~4, (2018), 1878--1916.

\bibitem[DP18b]{DPF18}
D\"obler, C. and Peccati, G.:
Fourth moment theorems on the Poisson space:
analytic statements via product formulae.
\emph{Electron. Commun. Probab.} \textbf{23} (2018),
 no.~91, 12 pp.

\bibitem[DVZ18]{DVZ18}
D\"obler, C., Vidotto, A. and Zheng, G.:
Fourth moment theorems on the Poisson space in any dimension.
\emph{Electron. J. Probab.} \textbf{23} (2018),
 no.~36, 27 pp.

\bibitem[ET14]{ET14}
Eichelsbacher, P. and Th\"ale, C.:
New Berry--Esseen bounds for non-linear functionals of Poisson random
measures.
\emph{Electron. J. Probab.} \textbf{19} (2014),
  no.~102, 25 pp.

 \bibitem[LRP13]{LRP13}
Lachi\`eze-Rey, R. and Peccati, G.:
Fine Gaussian fluctuations on the Poisson space, I:
contractions, cumulants and geometric random graphs.
\emph{Electron. J. Probab.} \textbf{18} (2013),
  no.~32, 32 pp.



\bibitem[Las16]{Last16}
Last, G.:
Stochastic analysis for Poisson processes.
In: G. Peccati and M. Reitzner (eds.),
\emph{Stochastic Analysis for Poisson Point Processes},
Bocconi \& Springer Series, vol.~7, pp.~1--36,
Springer, Cham, 2016.



\bibitem[LP11]{LP11}
Last, G. and Penrose, M.D.:
Poisson process Fock space representation, chaos expansion and
covariance inequalities.
\emph{Probab. Theory Related Fields} \textbf{150} (2011), 663--690.


\bibitem[LP18]{LP18}
Last, G. and Penrose, M.:
\emph{Lectures on the Poisson Process}.
Institute of Mathematical Statistics Textbooks, vol.~7,
Cambridge University Press, Cambridge, 2018.



\bibitem[Led12]{Led12}

 Ledoux, M.:
{\it Chaos of a Markov operator and the fourth moment condition},
Ann. Probab. {\bf 40} (2012), no.~6, 2439--2459.



\bibitem[NP09]{NP09}
Nourdin, I. and Peccati, G.:
Stein's method on Wiener chaos.
\emph{Probab. Theory Related Fields}
\textbf{145}, nos.~1--2, (2009), 75--118.


\bibitem[NP12]{bluebook}
Nourdin, I. and Peccati, G.:
\emph{Normal Approximations with Malliavin Calculus:
From Stein's Method to Universality}.
Cambridge Tracts in Mathematics, vol.~192,
Cambridge University Press, Cambridge, 2012.


\bibitem[NP05]{FMT}
Nualart, D. and Peccati, G.:
Central limit theorems for sequences of multiple stochastic integrals.
\emph{Ann. Probab.} \textbf{33}, no.~1, (2005), 177--193.

\bibitem[PSTU10]{PSTU10}
Peccati, G., Sol\'e, J.L., Taqqu, M.S. and Utzet, F.:
Stein's method and normal approximation of Poisson functionals.
\emph{Ann. Probab.} \textbf{38}, no.~2, (2010), 443--478.






\bibitem[PT08]{PT08}
Peccati, G. and Taqqu, M.S.:
Central limit theorems for double Poisson integrals.
\emph{Bernoulli} \textbf{14}, no.~3, (2008), 791--821.

\bibitem[PT05]{PT05}
Peccati, G. and Tudor, C.A.:
Gaussian limits for vector-valued multiple stochastic integrals.
In: \emph{S\'eminaire de Probabilit\'es XXXVIII},
Lecture Notes in Mathematics, vol.~1857,
pp.~247--262,
Springer, Berlin, 2005.




\bibitem[PZ10]{PZ10}
Peccati, G. and Zheng, C.:
Multi-dimensional Gaussian fluctuations on the Poisson space.
\emph{Electron. J. Probab.} \textbf{15} (2010),
  no.~48, 1487--1527.

\bibitem[PZ14]{PZ14}
Peccati, G. and Zheng, C.:
Universal Gaussian fluctuations on the discrete Poisson chaos.
\emph{Bernoulli} \textbf{20}, no.~2, (2014), 697--715.

\bibitem[Pri09]{Pri09}
Privault, N.:
{\it Stochastic Analysis in Discrete and Continuous Settings:
With Normal Martingales.}
Lecture Notes in Mathematics 1982,
Springer, Berlin, 2009.

\bibitem[Sch16]{Sch16}
Schulte, M.:
Normal approximation of Poisson functionals in Kolmogorov distance.
\emph{J. Theoret. Probab.} \textbf{29} (2016), 96--117.

\bibitem[Sur84]{Surg84}
Surgailis, D.:
On multiple Poisson stochastic integrals and associated Markov
semigroups.
\emph{Probab. Math. Statist.} \textbf{3}, no.~2, (1984), 217--239.


\bibitem[Zhe18]{Zheng18}
Zheng, G.:
\emph{Recent Developments Around the Malliavin--Stein Approach---Fourth
Moment Phenomena via Exchangeable Pairs}.
Ph.D. thesis, Universit\'e du Luxembourg, 2018.

\end{thebibliography}
\end{document}